\documentclass{amsart}

\usepackage{amsfonts}
\usepackage{amsmath}
\usepackage{amssymb}
\usepackage{latexsym}
\usepackage{amscd}
\usepackage{amsthm}
\usepackage{subfigure}
\usepackage{graphicx}
\usepackage[all]{xy}
\usepackage{color}
\usepackage{bm}

\newtheorem{thm}{Theorem}[section]
\newtheorem{prop}[thm]{Proposition}
\newtheorem{cor}[thm]{Corollary}
\newtheorem{lem}[thm]{Lemma}

\theoremstyle{definition}
\newtheorem{exam}[thm]{Example}
\newtheorem{rem}[thm]{Remark}

\newcommand{\Z}{\mathbb{Z}}
\newcommand{\G}{\Gamma}
\newcommand{\GH}{\hat{\Gamma}}
\newcommand{\SL}{\textrm{SL}}
\newcommand{\GL}{\textrm{GL}}

\author[N. Imoto]{Nao Imoto}

\author[R. Kobayashi]{Ryoma Kobayashi}
\address[R. Kobayashi]{
Department of General Education,\endgraf
National Institute of Technology, Ishikawa College,\endgraf
Tsubata, Ishikawa, 929-0392, Japan
}
\email{kobayashi\_ryoma@ishikawa-nct.ac.jp}

\thanks{\textit{Key words and phrases}. congruence subgroup, generator, presentation}

\thanks{The second author was supported by JSPS KAKENHI Grant Numbers JP19K14542 and JP22K13920.
}

\begin{document}

\title{Note on the level $d$ principal congruence subgroup of $\SL(n;\Z)$}

\maketitle

%%%%%%%%%%%%%%%%%%%%%%%%%%%%%%%%%%%%%%%%%%%%%%%%%%%%%%%%%%%%%%%%%%%%%%%%%%%%%%%%%%%%%%%%%%%%%%%%%%%%
%%%%%%%%%%%%%%%%%%%%%%%%%%%%%%%%%%%%%%%%%%%%%%%%%%%%%%%%%%%%%%%%%%%%%%%%%%%%%%%%%%%%%%%%%%%%%%%%%%%%
%%%%%%%%%%%%%%%%%%%%%%%%%%%%%%%%%%%%%%%%%%%%%%%%%%%%%%%%%%%%%%%%%%%%%%%%%%%%%%%%%%%%%%%%%%%%%%%%%%%%
%%%%%%%%%%%%%%%%%%%%%%%%%%%%%%%%%%%%%%%%%%%%%%%%%%%%%%%%%%%%%%%%%%%%%%%%%%%%%%%%%%%%%%%%%%%%%%%%%%%%
%%%%%%%%%%%%%%%%%%%%%%%%%%%%%%%%%%%%%%%%%%%%%%%%%%%%%%%%%%%%%%%%%%%%%%%%%%%%%%%%%%%%%%%%%%%%%%%%%%%%
%%%%%%%%%%%%%%%%%%%%%%%%%%%%%%%%%%%%%%%%%%%%%%%%%%%%%%%%%%%%%%%%%%%%%%%%%%%%%%%%%%%%%%%%%%%%%%%%%%%%
%%%%%%%%%%%%%%%%%%%%%%%%%%%%%%%%%%%%%%%%%%%%%%%%%%%%%%%%%%%%%%%%%%%%%%%%%%%%%%%%%%%%%%%%%%%%%%%%%%%%
%%%%%%%%%%%%%%%%%%%%%%%%%%%%%%%%%%%%%%%%%%%%%%%%%%%%%%%%%%%%%%%%%%%%%%%%%%%%%%%%%%%%%%%%%%%%%%%%%%%%
%%%%%%%%%%%%%%%%%%%%%%%%%%%%%%%%%%%%%%%%%%%%%%%%%%%%%%%%%%%%%%%%%%%%%%%%%%%%%%%%%%%%%%%%%%%%%%%%%%%%
%%%%%%%%%%%%%%%%%%%%%%%%%%%%%%%%%%%%%%%%%%%%%%%%%%%%%%%%%%%%%%%%%%%%%%%%%%%%%%%%%%%%%%%%%%%%%%%%%%%%
\begin{abstract}
The abelianization of the level $d$ principal congruence subgroup $\G_d(n)$ of $\SL(n;\Z)$ was determined by \cite{LS}.
By this result and a result of \cite{Ti}, we can obtain a minimal generating set for $\G_d(n)$.
In this paper, we give a minimal generating set for $\G_d(n)$ and determine the abelianization of $\G_d(n)$, without using the results of \cite{Ti} and \cite{LS}.
Moreover, we give three theorems about $\G_d(n)$.
\end{abstract}

%%%%%%%%%%%%%%%%%%%%%%%%%%%%%%%%%%%%%%%%%%%%%%%%%%%%%%%%%%%%%%%%%%%%%%%%%%%%%%%%%%%%%%%%%%%%%%%%%%%%
%%%%%%%%%%%%%%%%%%%%%%%%%%%%%%%%%%%%%%%%%%%%%%%%%%%%%%%%%%%%%%%%%%%%%%%%%%%%%%%%%%%%%%%%%%%%%%%%%%%%
%%%%%%%%%%%%%%%%%%%%%%%%%%%%%%%%%%%%%%%%%%%%%%%%%%%%%%%%%%%%%%%%%%%%%%%%%%%%%%%%%%%%%%%%%%%%%%%%%%%%
%%%%%%%%%%%%%%%%%%%%%%%%%%%%%%%%%%%%%%%%%%%%%%%%%%%%%%%%%%%%%%%%%%%%%%%%%%%%%%%%%%%%%%%%%%%%%%%%%%%%
%%%%%%%%%%%%%%%%%%%%%%%%%%%%%%%%%%%%%%%%%%%%%%%%%%%%%%%%%%%%%%%%%%%%%%%%%%%%%%%%%%%%%%%%%%%%%%%%%%%%
%%%%%%%%%%%%%%%%%%%%%%%%%%%%%%%%%%%%%%%%%%%%%%%%%%%%%%%%%%%%%%%%%%%%%%%%%%%%%%%%%%%%%%%%%%%%%%%%%%%%
%%%%%%%%%%%%%%%%%%%%%%%%%%%%%%%%%%%%%%%%%%%%%%%%%%%%%%%%%%%%%%%%%%%%%%%%%%%%%%%%%%%%%%%%%%%%%%%%%%%%
%%%%%%%%%%%%%%%%%%%%%%%%%%%%%%%%%%%%%%%%%%%%%%%%%%%%%%%%%%%%%%%%%%%%%%%%%%%%%%%%%%%%%%%%%%%%%%%%%%%%
%%%%%%%%%%%%%%%%%%%%%%%%%%%%%%%%%%%%%%%%%%%%%%%%%%%%%%%%%%%%%%%%%%%%%%%%%%%%%%%%%%%%%%%%%%%%%%%%%%%%
%%%%%%%%%%%%%%%%%%%%%%%%%%%%%%%%%%%%%%%%%%%%%%%%%%%%%%%%%%%%%%%%%%%%%%%%%%%%%%%%%%%%%%%%%%%%%%%%%%%%
\section{Introduction}

For integers $n\geq1$ and $d\geq2$, let $\G_d(n)$ denote the kernel of the natural homomorphism $\SL(n;\Z)\to\SL(n;\Z/d\Z)$, that is, $\G_d(n)$ is defined as
$$\G_d(n)=\{X\in\SL(n;\Z){\mid}X\equiv{I}\pmod{d}\}.$$
We define $\G_1(n)=\SL(n;\Z)$.
We call $\G_d(n)$ the \textit{level $d$ principal congruence subgroup} of $\SL(n;\Z)$.
Similarly, we define the level $d$ principal congruence subgroup $\GH_d(n)$ of $\GL(n;\Z)$ for $d\geq1$.
When $g\geq3$, for any $X\in\GH_d(n)$, we have $\det{X}=1$ by the definition.
Hence we see $\G_d(n)=\GH_d(n)$ when $d\geq3$.
On the other hand, $\G_d(n)$ is an index $2$ subgroup of $\GH_d(n)$ when $d=1$ or $2$.
Note that for any $X\in\GH_d(n)$, there is a matrix $A$ such that $X=dA+I$.

These groups have been well studied (for instance see \cite{BLS, Me, BMS}).
In addition, these are closely related to the mapping class group of a non-orientable surface (for instance see \cite{MP, GP, HK, KO, Ko2}).
Therefore, they play an important role in research on group theory, topology and related topics.

Lee-Szczarba~\cite{LS} determined the abelianization of $\G_d(n)$ as follows.

%%%%%%%%%%%%%%%%%%%%%%%%%%%%%%%%%%%%%%%%%%%%%%%%%%%%%%%%%%%%%%%%%%%%%%%%%%%%%%%%%%%%%%%%%%%%%%%%%%%%
\begin{thm}[\cite{LS}]\label{main-2}
 For $n\geq3$ and $d\geq1$, the abelianization of $\G_d(n)$ is isomorphic to $\left(\Z/d\Z\right)^{n^2-1}$.
\end{thm}
%%%%%%%%%%%%%%%%%%%%%%%%%%%%%%%%%%%%%%%%%%%%%%%%%%%%%%%%%%%%%%%%%%%%%%%%%%%%%%%%%%%%%%%%%%%%%%%%%%%%

For $1\leq{i,j}\leq{n}$ with $i\neq{}j$, let $e_{ij}\in\SL(n;\Z)$ be the matrix whose $(i,j)$ entry and the diagonal entries are $1$ and the other entries are $0$.
Lee-Szczarba~\cite{LS} constructed an epimorphism $\G_d(n)\to(\Z/d\Z)^{n^2-1}$ whose kernel is $\G_{d^2}(n)$ and showed that $\G_{d^2}(n)$ is equal to $[\G_d(n),\G_d(n)]$ for $d\ge2$ and $n\ge3$.
In addition, we can also see that the image of the epimorphism is generated by the natural projections of $e_{ij}^d$ and $e_{k1}e_{1k}^de_{k1}^{-1}$ for $1\leq{i,j}\leq{n}$ with $i\neq{}j$ and $2\leq{k}\leq{n}$.
Moreover, Tits~\cite{Ti} showed that $\G_{d^2}(n)$ is a subgroup of the subgroup of $\SL(n;\Z)$ generated by $e_{ij}^d$ for $1\le{i,j}\le{n}$ with $i\neq{}j$.
Therefore we have the following.

%%%%%%%%%%%%%%%%%%%%%%%%%%%%%%%%%%%%%%%%%%%%%%%%%%%%%%%%%%%%%%%%%%%%%%%%%%%%%%%%%%%%%%%%%%%%%%%%%%%%
\begin{thm}\label{main-1}
For $n\geq3$ and $d\geq1$, $\G_d(n)$ is generated by $e_{ij}^d$ and $e_{k1}e_{1k}^de_{k1}^{-1}$ for $1\leq{i,j}\leq{n}$ with $i\neq{}j$ and $2\leq{k}\leq{n}$.
\end{thm}
%%%%%%%%%%%%%%%%%%%%%%%%%%%%%%%%%%%%%%%%%%%%%%%%%%%%%%%%%%%%%%%%%%%%%%%%%%%%%%%%%%%%%%%%%%%%%%%%%%%%

By Theorem~\ref{main-1}, $\G_d(n)$ can be generated by $n^2-1$ elements.
Therefore from Theorem~\ref{main-2}, we have the following.

%%%%%%%%%%%%%%%%%%%%%%%%%%%%%%%%%%%%%%%%%%%%%%%%%%%%%%%%%%%%%%%%%%%%%%%%%%%%%%%%%%%%%%%%%%%%%%%%%%%%
\begin{cor}
The generating set for $\G_d(n)$ in Theorem~\ref{main-1} is minimal.
\end{cor}
%%%%%%%%%%%%%%%%%%%%%%%%%%%%%%%%%%%%%%%%%%%%%%%%%%%%%%%%%%%%%%%%%%%%%%%%%%%%%%%%%%%%%%%%%%%%%%%%%%%%

In this paper, we prove first Theorem~\ref{main-1} and then Theorem~\ref{main-2}, without using the results of \cite{Ti} and \cite{LS}.
In addition, we prove the following three theorems.

%%%%%%%%%%%%%%%%%%%%%%%%%%%%%%%%%%%%%%%%%%%%%%%%%%%%%%%%%%%%%%%%%%%%%%%%%%%%%%%%%%%%%%%%%%%%%%%%%%%%
\begin{thm}\label{main-3}
Let $E_{ij}=e_{ij}^2$ and $F_{kl}=e_{kl}^2\cdot{}e_{lk}^{-2}\cdot{}e_{lk}e_{kl}^2e_{lk}^{-1}$.
For $n\geq2$, $\G_2(n)$ is generated by $E_{ij}$ and $F_{1k}$ for $1\leq{i,j}\leq{n}$ with $i\neq{}j$ and $2\leq{k}\leq{n}$.
The defining relators are as follows.
\begin{enumerate}
\item	$F_{1i}^2$, $[E_{1i},F_{1i}]$, $[E_{i1},F_{1i}]$ when $n\geq2$,
\item	$[E_{ij},E_{ik}]$, $[E_{ij},E_{kj}]$, $[E_{ij},E_{jk}]E_{ik}^{-2}$, $(F_{1i}F_{1j})^2$, $(E_{1i}F_{1j})^2$, $(E_{ij}F_{1j})^2$, $(E_{i1}F_{1j})^2$, $(E_{ij}F_{1i})^2$ when $n\geq3$,
\item	$[E_{ij},E_{kl}]$ when $n\geq4$,
\item	$(E_{ji}^{-1}E_{ij}E_{kj}^{-1}E_{jk}E_{ik}^{-1}E_{ki})^2$ with $i<j<k$, when $n\geq3$,
\end{enumerate}
where indices $1$, $i$, $j$, $k$ and $l$ are distinct.
\end{thm}
%%%%%%%%%%%%%%%%%%%%%%%%%%%%%%%%%%%%%%%%%%%%%%%%%%%%%%%%%%%%%%%%%%%%%%%%%%%%%%%%%%%%%%%%%%%%%%%%%%%%

%%%%%%%%%%%%%%%%%%%%%%%%%%%%%%%%%%%%%%%%%%%%%%%%%%%%%%%%%%%%%%%%%%%%%%%%%%%%%%%%%%%%%%%%%%%%%%%%%%%%
\begin{thm}\label{main-4}
\begin{enumerate}
\item	When $d=3$ or $4$, $\G_d(2)$ is generated by $e_{21}^d$ and $e_{21}^me_{12}^de_{21}^{-m}$ for $0\le{m}\le{}d-1$.
\item	$\G_5(2)$ is generated by $e_{21}^5$, $e_{21}^me_{12}^5e_{21}^{-m}$ and $e_{21}^me_{12}^{\pm2}e_{21}^5e_{12}^{\mp2}e_{21}^{-m}$ for $0\le{m}\le4$.
\item	$\G_6(2)$ is generated by $e_{21}^6$, $e_{21}^me_{12}^6e_{21}^{-m}$ and $e_{21}^m[e_{21}^{\pm3},e_{12}^{\pm2}]e_{21}^{-m}$ for $0\le{m}\le5$.
\end{enumerate}
\end{thm}
%%%%%%%%%%%%%%%%%%%%%%%%%%%%%%%%%%%%%%%%%%%%%%%%%%%%%%%%%%%%%%%%%%%%%%%%%%%%%%%%%%%%%%%%%%%%%%%%%%%%

%%%%%%%%%%%%%%%%%%%%%%%%%%%%%%%%%%%%%%%%%%%%%%%%%%%%%%%%%%%%%%%%%%%%%%%%%%%%%%%%%%%%%%%%%%%%%%%%%%%%
\begin{thm}\label{main-5}
Let $l$, $m\ge1$ and $n\geq3$.
\begin{enumerate}
\item	$\displaystyle\G_{\gcd(l,m)}(n)/\G_m(n)$ is isomorphic to $\displaystyle\G_l(n)/\G_{\mathrm{lcm}(l,m)}(n)$.
\item	$\displaystyle\G_{\gcd(l,m)}(n)/\G_{\mathrm{lcm}(l,m)}(n)$ is isomorphic to
	\begin{eqnarray*}
	&&\G_{\gcd(l,m)}(n)/\G_l(n)\times\G_{\gcd(l,m)}(n)/\G_m(n),\\
	&&\G_l(n)/\G_{\mathrm{lcm}(l,m)}(n)\times\G_m(n)/\G_{\mathrm{lcm}(l,m)}(n).
	\end{eqnarray*}
\item	$\displaystyle\G_l(n)/\G_m(n)$ is isomorphic to $\left(\Z/\frac{m}{l}\Z\right)^{n^2-1}$, where $l$ and $m$ satisfy conditions $l\mid{m}$ and $m\mid{l^2}$.
\end{enumerate}
\end{thm}
%%%%%%%%%%%%%%%%%%%%%%%%%%%%%%%%%%%%%%%%%%%%%%%%%%%%%%%%%%%%%%%%%%%%%%%%%%%%%%%%%%%%%%%%%%%%%%%%%%%%

Proofs of Theorems~\ref{main-1}, \ref{main-2}, \ref{main-3}, \ref{main-4} and \ref{main-5} are given in Sections~\ref{thm-1}, \ref{thm-2}, \ref{thm-3}, \ref{thm-4} and \ref{thm-5}, respectively. 

%%%%%%%%%%%%%%%%%%%%%%%%%%%%%%%%%%%%%%%%%%%%%%%%%%%%%%%%%%%%%%%%%%%%%%%%%%%%%%%%%%%%%%%%%%%%%%%%%%%%
\begin{rem}\label{relation}
We can check that for any integers $s$ and $t$, following relations hold.
\begin{itemize}
\item	$e_{ij}^se_{ik}^t=e_{ik}^te_{ij}^s$,
\item	$e_{ij}^se_{kj}^t=e_{kj}^te_{ij}^s$,
\item	$e_{ij}^se_{kl}^t=e_{kl}^te_{ij}^s$,
\item	$e_{ij}^se_{jk}^t=e_{jk}^te_{ij}^se_{ik}^{st}$,
\item	$e_{ji}e_{ij}^se_{ji}^{-1}=e_{ij}e_{ji}^{-s}e_{ij}^{-1}$,
\end{itemize}
where indices $i$, $j$, $k$ and $l$ are distinct.
\end{rem}

%%%%%%%%%%%%%%%%%%%%%%%%%%%%%%%%%%%%%%%%%%%%%%%%%%%%%%%%%%%%%%%%%%%%%%%%%%%%%%%%%%%%%%%%%%%%%%%%%%%%
%%%%%%%%%%%%%%%%%%%%%%%%%%%%%%%%%%%%%%%%%%%%%%%%%%%%%%%%%%%%%%%%%%%%%%%%%%%%%%%%%%%%%%%%%%%%%%%%%%%%
%%%%%%%%%%%%%%%%%%%%%%%%%%%%%%%%%%%%%%%%%%%%%%%%%%%%%%%%%%%%%%%%%%%%%%%%%%%%%%%%%%%%%%%%%%%%%%%%%%%%
%%%%%%%%%%%%%%%%%%%%%%%%%%%%%%%%%%%%%%%%%%%%%%%%%%%%%%%%%%%%%%%%%%%%%%%%%%%%%%%%%%%%%%%%%%%%%%%%%%%%
%%%%%%%%%%%%%%%%%%%%%%%%%%%%%%%%%%%%%%%%%%%%%%%%%%%%%%%%%%%%%%%%%%%%%%%%%%%%%%%%%%%%%%%%%%%%%%%%%%%%
%%%%%%%%%%%%%%%%%%%%%%%%%%%%%%%%%%%%%%%%%%%%%%%%%%%%%%%%%%%%%%%%%%%%%%%%%%%%%%%%%%%%%%%%%%%%%%%%%%%%
%%%%%%%%%%%%%%%%%%%%%%%%%%%%%%%%%%%%%%%%%%%%%%%%%%%%%%%%%%%%%%%%%%%%%%%%%%%%%%%%%%%%%%%%%%%%%%%%%%%%
%%%%%%%%%%%%%%%%%%%%%%%%%%%%%%%%%%%%%%%%%%%%%%%%%%%%%%%%%%%%%%%%%%%%%%%%%%%%%%%%%%%%%%%%%%%%%%%%%%%%
%%%%%%%%%%%%%%%%%%%%%%%%%%%%%%%%%%%%%%%%%%%%%%%%%%%%%%%%%%%%%%%%%%%%%%%%%%%%%%%%%%%%%%%%%%%%%%%%%%%%
%%%%%%%%%%%%%%%%%%%%%%%%%%%%%%%%%%%%%%%%%%%%%%%%%%%%%%%%%%%%%%%%%%%%%%%%%%%%%%%%%%%%%%%%%%%%%%%%%%%%
\section{Proof of Theorem~\ref{main-1}}\label{thm-1}

Let $X_1$ and $X_2$ be
\begin{eqnarray*}
X_1&=&\{e_{ij}^d,~e_{k1}e_{1k}^de_{k1}^{-1}\mid1\le{i,j}\le{n}~\textrm{with}~i\ne{j},~2\le{k}\le{n}\},\\
X_2&=&\{e_{ji}^me_{ij}^de_{ji}^{-m}\mid1\le{i,j}\le{n}~\textrm{with}~i\ne{j},~m=0,1\},
\end{eqnarray*}
and $G_1$ and $G_2$ the subgroups of $\SL(n;\Z)$ generated by $X_1$ and $X_2$, respectively.
Since we have $X_1\subset{}X_2\subset\G_d(n)$, it is clear that $G_1\subset{}G_2\subset\G_d(n)$.
We prove that $G_2=\G_d(n)$ first and then $G_1\supset{}G_2$.

First, we prove the following lemma.

%%%%%%%%%%%%%%%%%%%%%%%%%%%%%%%%%%%%%%%%%%%%%%%%%%%%%%%%%%%%%%%%%%%%%%%%%%%%%%%%%%%%%%%%%%%%%%%%%%%%
\begin{lem}\label{md-m}
For any integer $m$, the following relations hold.
\begin{eqnarray*}
e_{ji}^me_{ij}^de_{ji}^{-m}&=&e_{jk}^me_{kj}^de_{jk}^{-m}\cdot{}e_{ij}^d(e_{ik}^d)^{-m}(e_{ik}e_{ki}^de_{ik}^{-1})^{-m}(e_{ji}^d)^{-m^2}(e_{jk}^d)^{m^2}(e_{ki}^d)^m(e_{kj}^d)^{-1},\\
e_{ji}^me_{ij}^de_{ji}^{-m}&=&e_{ik}^d(e_{jk}^d)^m(e_{kj}e_{jk}^de_{kj}^{-1})^{-m}(e_{ji}^d)^{-m^2}(e_{ki}^d)^{-m^2}(e_{ki}^{-m}e_{ik}^de_{ki}^m)^{-1}e_{ij}^d(e_{kj}^d)^{-m},
\end{eqnarray*}
where $k\neq{i,j}$.
\end{lem}
%%%%%%%%%%%%%%%%%%%%%%%%%%%%%%%%%%%%%%%%%%%%%%%%%%%%%%%%%%%%%%%%%%%%%%%%%%%%%%%%%%%%%%%%%%%%%%%%%%%%

\begin{proof}
By Remark~\ref{relation}, we calculate
\begin{eqnarray*}
e_{ji}^me_{ij}^de_{ji}^{-m}
&=&e_{ji}^m[e_{ik},e_{kj}^d]e_{ji}^{-m}\\
&=&[e_{ji}^me_{ik}e_{ji}^{-m},e_{ji}^me_{kj}^de_{ji}^{-m}]\\
&=&[e_{ik}e_{jk}^m,e_{kj}^de_{ki}^{-md}]\\
&=&(e_{jk}^me_{ik}e_{kj}^de_{ik}^{-1}e_{jk}^{-m})(e_{ik}e_{jk}^me_{ki}^{-md}e_{jk}^{-m}e_{ik}^{-1})e_{ki}^{md}e_{kj}^{-d}\\
&=&(e_{jk}^me_{kj}^de_{ij}^de_{jk}^{-m})(e_{ik}e_{ki}^{-md}e_{ji}^{-m^2d}e_{ik}^{-1})e_{ki}^{md}e_{kj}^{-d}\\
&=&e_{jk}^me_{kj}^de_{jk}^{-m}\cdot{}e_{jk}^me_{ij}^de_{jk}^{-m}\cdot{}e_{ik}e_{ki}^{-md}e_{ik}^{-1}\cdot{}e_{ik}e_{ji}^{-m^2d}e_{ik}^{-1}\cdot{}e_{ki}^{md}\cdot{}e_{kj}^{-d}\\
&=&e_{jk}^me_{kj}^de_{jk}^{-m}\cdot{}e_{ij}^d\cdot{}e_{ik}^{-md}\cdot{}e_{ik}e_{ki}^{-md}e_{ik}^{-1}\cdot{}e_{ji}^{-m^2d}\cdot{}e_{jk}^{m^2d}\cdot{}e_{ki}^{md}\cdot{}e_{kj}^{-d}\\
&=&e_{jk}^me_{kj}^de_{jk}^{-m}\cdot{}e_{ij}^d(e_{ik}^d)^{-m}(e_{ik}e_{ki}^de_{ik}^{-1})^{-m}(e_{ji}^d)^{-m^2}(e_{jk}^d)^{m^2}(e_{ki}^d)^m(e_{kj}^d)^{-1},
\end{eqnarray*}
\begin{eqnarray*}
e_{ji}^me_{ij}^de_{ji}^{-m}
&=&e_{ji}^m[e_{ik}^d,e_{kj}]e_{ji}^{-m}\\
&=&[e_{ji}^me_{ik}^de_{ji}^{-m},e_{ji}^me_{kj}e_{ji}^{-m}]\\
&=&[e_{ik}^de_{jk}^{md},e_{kj}e_{ki}^{-m}]\\
&=&e_{ik}^de_{jk}^{md}(e_{kj}e_{ki}^{-m}e_{jk}^{-md}e_{ki}^me_{kj}^{-1})(e_{ki}^{-m}e_{kj}e_{ik}^{-d}e_{kj}^{-1}e_{ki}^m)\\
&=&e_{ik}^de_{jk}^{md}(e_{kj}e_{jk}^{-md}e_{ji}^{-m^2d}e_{kj}^{-1})(e_{ki}^{-m}e_{ik}^{-d}e_{ij}^de_{ki}^m)\\
&=&e_{ik}^d\cdot{}e_{jk}^{md}\cdot{}e_{kj}e_{jk}^{-md}e_{kj}^{-1}\cdot{}e_{kj}e_{ji}^{-m^2d}e_{kj}^{-1}\cdot{}e_{ki}^{-m}e_{ik}^{-d}e_{ki}^m\cdot{}e_{ki}^{-m}e_{ij}^de_{ki}^m\\
&=&e_{ik}^d\cdot{}e_{jk}^{md}\cdot{}e_{kj}e_{jk}^{-md}e_{kj}^{-1}\cdot{}e_{ji}^{-m^2d}e_{ki}^{-m^2d}\cdot{}e_{ki}^{-m}e_{ik}^{-d}e_{ki}^m\cdot{}e_{ij}^de_{kj}^{-md}\\
&=&e_{ik}^d(e_{jk}^d)^m(e_{kj}e_{jk}^de_{kj}^{-1})^{-m}(e_{ji}^d)^{-m^2}(e_{ki}^d)^{-m^2}(e_{ki}^{-m}e_{ik}^de_{ki}^m)^{-1}e_{ij}^d(e_{kj}^d)^{-m}.
\end{eqnarray*}
Thus we get the claim.
\end{proof}

Next, we prove the following lemma.

%%%%%%%%%%%%%%%%%%%%%%%%%%%%%%%%%%%%%%%%%%%%%%%%%%%%%%%%%%%%%%%%%%%%%%%%%%%%%%%%%%%%%%%%%%%%%%%%%%%%
\begin{lem}\label{2d-2}
$e_{ji}^2e_{ij}^de_{ji}^{-2}$ and $e_{ji}^{-1}e_{ij}^de_{ji}$ are in $G_2$.
\end{lem}
%%%%%%%%%%%%%%%%%%%%%%%%%%%%%%%%%%%%%%%%%%%%%%%%%%%%%%%%%%%%%%%%%%%%%%%%%%%%%%%%%%%%%%%%%%%%%%%%%%%%

\begin{proof}
By Remark~\ref{relation}, we have
\begin{eqnarray*}
e_{ji}e_{jk}e_{kj}^de_{jk}^{-1}e_{ji}^{-1}&=&e_{jk}e_{ji}e_{kj}^de_{ji}^{-1}e_{jk}^{-1}=e_{jk}e_{kj}^de_{ki}^{-d}e_{jk}^{-1}=e_{jk}e_{kj}^de_{jk}^{-1}\cdot{}e_{jk}e_{ki}^{-d}e_{jk}^{-1}\\&=&e_{jk}e_{kj}^de_{jk}^{-1}\cdot{}e_{ki}^{-d}e_{ji}^{-d},\\
e_{ji}e_{ik}^de_{ji}^{-1}&=&e_{ik}^de_{jk}^d,\\
e_{ji}e_{ik}e_{ki}^de_{ik}^{-1}e_{ji}^{-1}&=&e_{ik}e_{ji}e_{jk}e_{ki}^de_{jk}^{-1}e_{ji}^{-1}e_{ik}^{-1}=e_{ik}e_{ji}e_{ki}^de_{ji}^de_{ji}^{-1}e_{ik}^{-1}=e_{ik}e_{ki}^de_{ji}^de_{ik}^{-1}\\&=&e_{ik}e_{ki}^de_{ik}^{-1}\cdot{}e_{ik}e_{ji}^de_{ik}^{-1}=e_{ik}e_{ki}^de_{ik}^{-1}\cdot{}e_{ji}^de_{jk}^{-d},\\
e_{ji}e_{jk}^de_{ji}^{-1}&=&e_{jk}^d,\\
e_{ji}e_{ki}^de_{ji}^{-1}&=&e_{ki}^d,\\
e_{ji}e_{kj}^de_{ji}^{-1}&=&e_{kj}^de_{ki}^{-d}.
\end{eqnarray*}
Hence by Lemma~\ref{md-m}, we have
\begin{eqnarray*}
e_{ji}^2e_{ij}^de_{ji}^{-2}
&=&e_{ji}(e_{ji}e_{ij}^de_{ji}^{-1})e_{ji}\\
&=&e_{ji}(e_{jk}e_{kj}^de_{jk}^{-1}\cdot{}e_{ij}^d(e_{ik}^d)^{-1}(e_{ik}e_{ki}^de_{ik}^{-1})^{-1}(e_{ji}^d)^{-1}e_{jk}^de_{ki}^d(e_{kj}^d)^{-1})e_{ji}^{-1}\\
&=&
(e_{jk}e_{kj}^de_{jk}^{-1}\cdot(e_{ki}^d)^{-1}\cdot(e_{ji}^d)^{-1})\cdot{}
e_{ji}e_{ij}^de_{ji}^{-1}\cdot{}
(e_{ik}^d\cdot{}e_{jk}^d)^{-1}\\
&&(e_{ik}e_{ki}^de_{ik}^{-1}\cdot{}e_{ji}^d\cdot(e_{jk}^d)^{-1})^{-1}\cdot{}
(e_{ji}^d)^{-1}\cdot{}
e_{jk}^d\cdot{}
e_{ki}^d\cdot{}
(e_{kj}^d\cdot(e_{ki}^d)^{-1})^{-1}\\
&\in&G_2.
\end{eqnarray*}
In addition, by Lemma~\ref{md-m}, we have
$$e_{ji}^{-1}e_{ij}^de_{ji}=e_{ik}^d\cdot(e_{jk}^d)^{-1}\cdot{}e_{kj}e_{jk}^de_{kj}^{-1}\cdot{}(e_{ji}^d)^{-1}\cdot(e_{ki}^d)^{-1}\cdot(e_{ki}e_{ik}^de_{ki}^{-1})^{-1}\cdot{}e_{ij}^d\cdot{}e_{kj}^d\in{G_2}.$$
Thus we get the claim.
\end{proof}

We now prove Theorem~\ref{main-1}.

\begin{proof}[Proof of Theorem~\ref{main-1}]

First, we show $G_2=\G_d(n)$.
It is well known that $\SL(n;\Z)$ is generated by $e_{ij}$ for $1\leq{i,j}\leq{n}$ with $i\neq{}j$.
In addition, Bass-Milnor-Serre \cite{BMS} showed that for $n\geq3$ and $d\geq2$, $\G_d(n)$ is normally generated by $e_{ij}^d$ for $1\leq{i,j}\leq{n}$ with $i\neq{}j$, in $\SL(n;\Z)$.
Hence in order to show $G_2=\G_d(n)$, it suffices to show that for any $e_{i^\prime{}j^\prime}$ and $x\in{X_2}$, $e_{i^\prime{}j^\prime}^{\pm1}xe_{i^\prime{}j^\prime}^{\mp1}$ is in $G_2$.
Fix $x=e_{ji}^me_{ij}^de_{ji}^{-m}$, where $m=0$ or $1$.
In the proof, we use Remark~\ref{relation}.
Let $(i^\prime,j^\prime)=(k,l)$, where $k$, $l\neq{}i$, $j$.
Then we have
$$e_{i^\prime{}j^\prime}^{\pm1}xe_{i^\prime{}j^\prime}^{\mp1}=x\in{G_2}.$$
Let $(i^\prime,j^\prime)=(i,k)$, where $k\neq{}i$, $j$.
Then we have
\begin{eqnarray*}
e_{i^\prime{}j^\prime}^{\pm1}xe_{i^\prime{}j^\prime}^{\mp1}
&=&
e_{ik}^{\pm1}e_{ji}^me_{ij}^de_{ji}^{-m}e_{ik}^{\mp1}\\
&=&
e_{ji}^me_{ik}^{\pm1}e_{jk}^{\mp{m}}e_{ij}^de_{jk}^{\pm{m}}e_{ik}^{\mp1}e_{ji}^{-m}\\
&=&
e_{ji}^me_{jk}^{\mp{m}}e_{ij}^de_{jk}^{\pm{m}}e_{ji}^{-m}\\
&=&
e_{ji}^me_{ij}^de_{ik}^{\pm{md}}e_{ji}^{-m}\\
&=&
e_{ji}^me_{ij}^de_{ji}^{-m}\cdot{}e_{ji}^me_{ik}^{\pm{md}}e_{ji}^{-m}\\
&=&
e_{ji}^me_{ij}^de_{ji}^{-m}\cdot{}e_{ik}^{\pm{md}}e_{jk}^{\pm{m^2d}}\\
&=&
e_{ji}^me_{ij}^de_{ji}^{-m}\cdot{}(e_{ik}^d)^{\pm{m}}\cdot(e_{jk}^d)^{\pm{m^2}}\\
&\in&
G_2.
\end{eqnarray*}
Let $(i^\prime,j^\prime)=(k,i)$, where $k\neq{}i$, $j$.
Then we have
\begin{eqnarray*}
e_{i^\prime{}j^\prime}^{\pm1}xe_{i^\prime{}j^\prime}^{\mp1}
&=&
e_{ki}^{\pm1}e_{ji}^me_{ij}^de_{ji}^{-m}e_{ki}^{\mp1}\\
&=&
e_{ji}^me_{ki}^{\pm1}e_{ij}^de_{ki}^{\mp1}e_{ji}^{-m}\\
&=&
e_{ji}^me_{ij}^de_{kj}^{\pm{d}}e_{ji}^{-m}\\
&=&
e_{ji}^me_{ij}^de_{ji}^{-m}\cdot{}e_{ji}^me_{kj}^{\pm{d}}e_{ji}^{-m}\\
&=&
e_{ji}^me_{ij}^de_{ji}^{-m}\cdot{}e_{kj}^{\pm{d}}e_{ki}^{\mp{md}}\\
&=&
e_{ji}^me_{ij}^de_{ji}^{-m}\cdot(e_{kj}^d)^{\pm1}\cdot(e_{ki}^d)^{\mp{m}}\\
&\in&
G_2.
\end{eqnarray*}
Let $(i^\prime,j^\prime)=(k,j)$, where $k\neq{}i$, $j$.
Then we have
\begin{eqnarray*}
e_{i^\prime{}j^\prime}^{\pm1}xe_{i^\prime{}j^\prime}^{\mp1}
&=&
e_{kj}^{\pm1}e_{ji}^me_{ij}^de_{ji}^{-m}e_{kj}^{\mp1}\\
&=&
e_{ji}^me_{kj}^{\pm1}e_{ki}^{\pm{m}}e_{ij}^de_{ki}^{\mp{m}}e_{kj}^{\mp1}e_{ji}^{-m}\\
&=&
e_{ji}^me_{ki}^{\pm{m}}e_{ij}^de_{ki}^{\mp{m}}e_{ji}^{-m}\\
&=&
e_{ji}^me_{ij}^de_{kj}^{\pm{md}}e_{ji}^{-m}\\
&=&
e_{ji}^me_{ij}^de_{ji}^{-m}\cdot{}e_{ji}^me_{kj}^{\pm{md}}e_{ji}^{-m}\\
&=&
e_{ji}^me_{ij}^de_{ji}^{-m}\cdot{}e_{kj}^{\pm{md}}e_{ki}^{\mp{m^2d}}\\
&=&
e_{ji}^me_{ij}^de_{ji}^{-m}\cdot{}(e_{kj}^d)^{\pm{m}}\cdot(e_{ki}^d)^{\mp{m^2}}\\
&\in&
G_2.
\end{eqnarray*}
Let $(i^\prime,j^\prime)=(j,k)$, where $k\neq{}i$, $j$.
Then we have
\begin{eqnarray*}
e_{i^\prime{}j^\prime}^{\pm1}xe_{i^\prime{}j^\prime}^{\mp1}
&=&
e_{jk}^{\pm1}e_{ji}^me_{ij}^de_{ji}^{-m}e_{jk}^{\mp1}\\
&=&
e_{ji}^me_{jk}^{\pm1}e_{ij}^de_{jk}^{\mp1}e_{ji}^{-m}\\
&=&
e_{ji}^me_{ij}^de_{ik}^{\mp{d}}e_{ji}^{-m}\\
&=&
e_{ji}^me_{ij}^de_{ji}^{-m}\cdot{}e_{ji}^me_{ik}^{\mp{d}}e_{ji}^{-m}\\
&=&
e_{ji}^me_{ij}^de_{ji}^{-m}\cdot{}e_{ik}^{\mp{d}}e_{jk}^{\mp{md}}\\
&=&
e_{ji}^me_{ij}^de_{ji}^{-m}\cdot(e_{ik}^d)^{\mp1}\cdot(e_{jk}^d)^{\mp{m}}\\
&\in&
G_2.
\end{eqnarray*}
Let $(i^\prime,j^\prime)=(j,i)$.
Then we have
$$e_{i^\prime{}j^\prime}^{\pm1}xe_{i^\prime{}j^\prime}^{\mp1}=e_{ji}^{\pm1}e_{ji}^me_{ij}^de_{ji}^{-m}e_{ji}^{\mp1}=e_{ji}^{m\pm1}e_{ij}^de_{ji}^{-m\mp1}.$$
Since $m\pm1$ is either $0$, $1$, $2$ or $-1$, by Lemma~\ref{2d-2}, it is in $G_2$.
Let $(i^\prime,j^\prime)=(i,j)$ and $k\neq{}i$, $j$.
Then we have
\begin{eqnarray*}
e_{i^\prime{}j^\prime}^{\pm1}xe_{i^\prime{}j^\prime}^{\mp1}
&=&
e_{ij}^{\pm1}e_{ji}^me_{ij}^de_{ji}^{-m}e_{ij}^{\mp1}\\
&=&
e_{ij}^{\pm1}e_{ji}^m[e_{ik}^d,e_{kj}]e_{ji}^{-m}e_{ij}^{\mp1}\\
&=&
[e_{ij}^{\pm1}e_{ji}^me_{ik}^de_{ji}^{-m}e_{ij}^{\mp1},e_{ij}^{\pm1}e_{ji}^me_{kj}e_{ji}^{-m}e_{ij}^{\mp1}]\\
&=&
[e_{ij}^{\pm1}e_{ik}^de_{jk}^{md}e_{ij}^{\mp1},e_{ij}^{\pm1}e_{kj}e_{ki}^{-m}e_{ij}^{\mp1}]\\
&=&
[e_{ik}^d\cdot{}e_{ij}^{\pm1}e_{jk}^{md}e_{ij}^{\mp1},e_{kj}\cdot{}e_{ij}^{\pm1}e_{ki}^{-m}e_{ij}^{\mp1}]\\
&=&
[e_{ik}^de_{jk}^{md}e_{ik}^{\pm{md}},e_{kj}e_{ki}^{-m}e_{kj}^{\pm{m}}]\\
&=&
[e_{ik}^{(1\pm{m})d}e_{jk}^{md},e_{kj}^{1\pm{m}}e_{ki}^{-m}]\\
&=&
e_{ik}^{(1\pm{m})d}e_{jk}^{md}\cdot{}e_{kj}^{1\pm{m}}e_{ki}^{-m}\cdot{}e_{jk}^{-md}e_{ik}^{-(1\pm{m})d}\cdot{}e_{ki}^me_{kj}^{-(1\pm{m})}\\
&=&
e_{ik}^{(1\pm{m})d}e_{jk}^{md}\cdot{}e_{kj}^{1\pm{m}}e_{ki}^{-m}e_{jk}^{-md}e_{ki}^me_{kj}^{-(1\pm{m})}\cdot{}e_{ki}^{-m}e_{kj}^{1\pm{m}}e_{ik}^{-(1\pm{m})d}e_{kj}^{-(1\pm{m})}e_{ki}^m\\
&=&
e_{ik}^{(1\pm{m})d}e_{jk}^{md}\cdot{}e_{kj}^{1\pm{m}}e_{jk}^{-md}e_{ji}^{-m^2d}e_{kj}^{-(1\pm{m})}\cdot{}e_{ki}^{-m}e_{ik}^{-(1\pm{m})d}e_{ij}^{(1\pm{m})^2d}e_{ki}^m\\
&=&
e_{ik}^{(1\pm{m})d}e_{jk}^{md}\cdot{}e_{kj}^{1\pm{m}}e_{jk}^{-md}e_{kj}^{-(1\pm{m})}\cdot{}e_{kj}^{1\pm{m}}e_{ji}^{-m^2d}e_{kj}^{-(1\pm{m})}\\
&&
e_{ki}^{-m}e_{ik}^{-(1\pm{m})d}e_{ki}^m\cdot{}e_{ki}^{-m}e_{ij}^{(1\pm{m})^2d}e_{ki}^m\\
&=&
e_{ik}^{(1\pm{m})d}e_{jk}^{md}\cdot{}e_{kj}^{1\pm{m}}e_{jk}^{-md}e_{kj}^{-(1\pm{m})}\cdot{}e_{ji}^{-m^2d}e_{ki}^{-m^2(1\pm{m})d}\\
&&
e_{ki}^{-m}e_{ik}^{-(1\pm{m})d}e_{ki}^m\cdot{}e_{ij}^{(1\pm{m})^2d}e_{kj}^{-m(1\pm{m})^2d}\\
&=&
(e_{ik}^d)^{1\pm{m}}\cdot(e_{jk}^d)^m\cdot(e_{kj}^{1\pm{m}}e_{jk}^de_{kj}^{-(1\pm{m})})^{-m}\cdot(e_{ji}^d)^{-m^2}\cdot(e_{ki}^d)^{-m^2(1\pm{m})}\\
&&
(e_{ki}^{-m}e_{ik}^de_{ki}^{m})^{-(1\pm{m})}\cdot{}(e_{ij}^d)^{(1\pm{m})^2}\cdot(e_{kj}^d)^{-m(1\pm{m})^2}.
\end{eqnarray*}
Since $1\pm{m}$ is either $0$, $1$ or $2$, and $-m$ is either $0$ or $-1$, by Lemma~\ref{2d-2}, $e_{kj}^{1\pm{m}}e_{jk}^de_{kj}^{-(1\pm{m})}$ and $e_{ki}^{-m}e_{ik}^de_{ki}^{m}$ are in $G_2$, and so is $e_{i^\prime{}j^\prime}^{\pm1}xe_{i^\prime{}j^\prime}^{\mp1}$.
Therefore, for any indices $i^\prime$ and $j^\prime$, we conclude that $e_{i^\prime{}j^\prime}^{\pm1}xe_{i^\prime{}j^\prime}^{\mp1}$ is in $G_2$, and so $G_2=\G_d(n)$.

Next, we show $G_1\supset{}G_2$.
It suffices to show that $e_{ji}e_{ij}^de_{ji}^{-1}\in{X_2}$ is in $G_1$.
Clearly $e_{j1}e_{1j}^de_{j1}^{-1}$ is in $G_1$.
In addition, by Remark~\ref{relation}, $e_{1i}e_{i1}^de_{1i}^{-1}=e_{i1}e_{1i}^{-d}e_{i1}^{-1}=(e_{i1}e_{1i}^de_{i1}^{-1})^{-1}$ is in $G_1$.
Suppose $i,j\ge2$.
By Lemma~\ref{md-m}, we have
$$e_{ji}e_{ij}^de_{ji}^{-1}=e_{j1}e_{1j}^de_{j1}^{-1}\cdot{}e_{ij}^d\cdot(e_{i1}^d)^{-1}\cdot(e_{i1}e_{1i}^de_{i1}^{-1})^{-1}\cdot(e_{ji}^d)^{-1}\cdot{}e_{j1}^d\cdot{}e_{1i}^d\cdot(e_{1j}^d)^{-1}\in{G_1}.$$
Therefore we conclude $G_1\supset{}G_2$.

Thus we complete the proof.
\end{proof}

%%%%%%%%%%%%%%%%%%%%%%%%%%%%%%%%%%%%%%%%%%%%%%%%%%%%%%%%%%%%%%%%%%%%%%%%%%%%%%%%%%%%%%%%%%%%%%%%%%%%
%%%%%%%%%%%%%%%%%%%%%%%%%%%%%%%%%%%%%%%%%%%%%%%%%%%%%%%%%%%%%%%%%%%%%%%%%%%%%%%%%%%%%%%%%%%%%%%%%%%%
%%%%%%%%%%%%%%%%%%%%%%%%%%%%%%%%%%%%%%%%%%%%%%%%%%%%%%%%%%%%%%%%%%%%%%%%%%%%%%%%%%%%%%%%%%%%%%%%%%%%
%%%%%%%%%%%%%%%%%%%%%%%%%%%%%%%%%%%%%%%%%%%%%%%%%%%%%%%%%%%%%%%%%%%%%%%%%%%%%%%%%%%%%%%%%%%%%%%%%%%%
%%%%%%%%%%%%%%%%%%%%%%%%%%%%%%%%%%%%%%%%%%%%%%%%%%%%%%%%%%%%%%%%%%%%%%%%%%%%%%%%%%%%%%%%%%%%%%%%%%%%
%%%%%%%%%%%%%%%%%%%%%%%%%%%%%%%%%%%%%%%%%%%%%%%%%%%%%%%%%%%%%%%%%%%%%%%%%%%%%%%%%%%%%%%%%%%%%%%%%%%%
%%%%%%%%%%%%%%%%%%%%%%%%%%%%%%%%%%%%%%%%%%%%%%%%%%%%%%%%%%%%%%%%%%%%%%%%%%%%%%%%%%%%%%%%%%%%%%%%%%%%
%%%%%%%%%%%%%%%%%%%%%%%%%%%%%%%%%%%%%%%%%%%%%%%%%%%%%%%%%%%%%%%%%%%%%%%%%%%%%%%%%%%%%%%%%%%%%%%%%%%%
%%%%%%%%%%%%%%%%%%%%%%%%%%%%%%%%%%%%%%%%%%%%%%%%%%%%%%%%%%%%%%%%%%%%%%%%%%%%%%%%%%%%%%%%%%%%%%%%%%%%
%%%%%%%%%%%%%%%%%%%%%%%%%%%%%%%%%%%%%%%%%%%%%%%%%%%%%%%%%%%%%%%%%%%%%%%%%%%%%%%%%%%%%%%%%%%%%%%%%%%%
\section{Proof of Theorem~\ref{main-2}}\label{thm-2}

First, we show the following proposition.

\begin{prop}\label{[G,G]}
For $n\geq3$, $d_1\geq1$ and $d_2\geq1$, we have
$$[\G_{d_1}(n),\G_{d_2}(n)]=[\G_{d_1}(n),\GH_{d_2}(n)]=[\GH_{d_1}(n),\GH_{d_2}(n)]=\G_{d_1d_2}(n).$$
\end{prop}

\begin{proof}
For any normal generator $e_{ij}^{d_1d_2}\in\G_{d_1d_2}(n)$ appeared at the beginning of the proof of Theorem~\ref{main-1} , by Remark~\ref{relation}, we have that $e_{ij}^{d_1d_2}=[e_{ik}^{d_1},e_{kj}^{d_2}]$ is in $[\G_{d_1}(n),\G_{d_2(n)}]$, and hence it follows that $\G_{d_1d_2}(n)\subset[\G_{d_1}(n),\G_{d_2(n)}]$.
Since we have $[\G_{d_1}(n),\G_{d_2(n)}]\subset[\G_{d_1}(n),\GH_{d_2(n)}]\subset[\GH_{d_1}(n),\GH_{d_2(n)}]$ clearly, we show $[\GH_{d_1}(n),\GH_{d_2(n)}]\subset\G_{d_1d_2}(n)$.
For any $[X,Y]\in[\GH_{d_1}(n),\GH_{d_2(n)}]$, there are matrices $A$, $\tilde{A}$, $B$, $\tilde{B}$ such that $X=d_1A+I$, $X^{-1}=d_1\tilde{A}+I$, $Y=d_2B+I$ and $Y^{-1}=d_2\tilde{B}+I$.
Note that $\det[X,Y]=1$.
We calculate
\begin{eqnarray*}
[X,Y]
&=&
(d_1A+I)(d_2B+I)(d_1\tilde{A}+I)(d_2\tilde{B}+I)\\
&\equiv&
(d_1A+d_2B+I)(d_1\tilde{A}+d_2\tilde{B}+I)\pmod{d_1d_2}\\
&\equiv&
d_1^2A\tilde{A}+d_1A+d_2^2B\tilde{B}+d_2B+d_1\tilde{A}+d_2\tilde{B}+I\pmod{d_1d_2}\\
&=&
(d_1^2A\tilde{A}+d_1A+d_1\tilde{A}+I)+(d_2^2B\tilde{B}+d_2B+d_2\tilde{B}+I)-I\\
&=&
XX^{-1}+YY^{-1}-I\\
&=&
I.
\end{eqnarray*}
Therefore $[X,Y]$ is in $\G_{d_1d_2}(n)$, and hence it follows that $[\GH_{d_1}(n),\GH_{d_2}(n)]\subset\G_{d_1d_2}(n)$.

Thus we get the claim.
\end{proof}

Let $f_k^d=e_{k1}^d(e_{k1}e_{1k}^de_{k1}^{-1})(e_{1k}^d)^{-1}$ and
$$X=\{e_{ij}^d,f_k^d\mid1\le{i,j}\le{n}~\textrm{with}~i\ne{j},~2\le{k}\le{n}\}.$$
Next, we show the following lemma.

\begin{lem}\label{generator}
For $n\geq3$ and $d\geq1$, $\G_d(n)$ is generated by $X$.
\end{lem}

\begin{proof}
Let $G$ be the subgroup of $\SL(n;\Z)$ generated by $X$.
Since we have $X\subset\G_d(n)$, it is clear that $G\subset\G_d(n)$.
We prove $G\supset\G_d(n)$.
Remember the finite generating set $X_1$ of $\G_d(n)$ in Theorem~\ref{main-1}.
$e_{ij}^d\in{X_1}$ is in $X$, and hence it is in $G$.
$e_{k1}e_{1k}^de_{k1}^{-1}\in{X_1}$ is described as $e_{k1}e_{1k}^de_{k1}^{-1}=(e_{k1}^d)^{-1}f_k^de_{1k}^d$, and hence it is in $G$.
Therefore $G\supset\G_d(n)$.
Thus we get the claim.
\end{proof}

\begin{rem}\label{f_k}
Modulo $d^2$, $f_k^d$ is the matrix whose $(1,1)$ entry is $1-d$, $(k,k)$ entry is $1+d$, the other diagonal entries are $1$ and the other entries are $0$.
\end{rem}

We now prove Theorem~\ref{main-2}.

\begin{proof}[Proof of Theorem~\ref{main-2}]
By Proposition~\ref{[G,G]}, we have $\G_d(n)/[\G_d(n),\G_d(n)]=\G_d(n)/\G_{d^2}(n)$.
For any generator $x\in{X_1}$ of $\G_d(n)$ in Theorem~\ref{main-1}, since $x^d$ is in $\G_{d^2}(n)$, it follows that $\G_d(n)/[\G_d(n),\G_d(n)]$ is a $\Z/d\Z$ module.
Hence the natural homomorphism
$$\Phi:\bigoplus_{x\in{X}}\Z/d\Z[x]\to\G_d(n)/[\G_d(n),\G_d(n)]$$
is well-defined and surjective, where $X$ is the generating set for $\G_d(n)$ in Lemma~\ref{generator}.
Note that $:\bigoplus_{x\in{X}}\Z/d\Z[x]$ is isomorphic to $\left(\Z/d\Z\right)^{n^2-1}$.
We show $\ker\Phi=1$.

For any $g\in\bigoplus_{x\in{X}}\Z/d\Z[x]$, there are integers $0\le{m_{ij},m_k}\le{d-1}$ for $1\le{i,j}\le{n}$ with $i\ne{}j$ and $2\le{k}\le{n}$ such that
$$g=\prod_{1\le{i,j}\le{n}~(i\ne{j})}[e_{ij}^d]^{m_{ij}}\prod_{2\le{k}\le{n}}[f_k^d]^{m_k}.$$
We note that
\begin{eqnarray*}
\begin{pmatrix}\bm{a}_1&\cdots&\bm{a}_n\end{pmatrix}e_{ij}^d&=&\begin{pmatrix}\bm{a}_1&\cdots&\bm{a}_{j-1}&\bm{a}_j+d\bm{a}_i&\bm{a}_{j+1}&\cdots&\bm{a}_n\end{pmatrix},\\
\begin{pmatrix}\bm{a}_1&\cdots&\bm{a}_n\end{pmatrix}f_k^d&=&\begin{pmatrix}(1-d)\bm{a}_1&\bm{a}_2&\cdots&\bm{a}_{k-1}&(1+d)\bm{a}_k&\bm{a}_{k+1}&\cdots&\bm{a}_n\end{pmatrix}
\end{eqnarray*}
modulo $d^2$ by Remark~\ref{f_k}.
We see
\begin{eqnarray*}
\prod_{1\le{i,j}\le{n}~(i\ne{j})}(e_{ij}^d)^{m_{ij}}
&=&
\begin{pmatrix}
1&m_{12}d&\cdots&m_{1n}d\\
m_{21}d&1&\ddots&\vdots\\
\vdots&\ddots&\ddots&m_{n-1\:n}d\\
m_{n1}d&\cdots&m_{n\:n-1}d&1
\end{pmatrix}
\pmod{d^2},\\
\prod_{2\le{k}\le{n}}(f_k^d)^{m_k}
&=&
\begin{pmatrix}
1-md&0&\cdots&0\\
0&1+m_2d&\ddots&\vdots\\
\vdots&\ddots&\ddots&0\\
0&\cdots&0&1+m_kd
\end{pmatrix}
\pmod{d^2},
\end{eqnarray*}
where $m=m_2+\cdots+m_k$, and hence
$$\Phi(g)=
\begin{pmatrix}
1-md&m_{12}d&\cdots&m_{1n}d\\
m_{21}d&1+m_2d&\ddots&\vdots\\
\vdots&\ddots&\ddots&m_{n-1\:n}d\\
m_{n1}d&\cdots&m_{n\:n-1}d&1+m_kd
\end{pmatrix}
\pmod{d^2}.
$$
Therefore, $g$ is in $\ker\Phi$ if and only if $m_{ij}=m_k=0$ for $1\le{i,j}\le{n}$ with $i\ne{}j$ and $2\le{k}\le{n}$.
So we conclude $\ker\Phi=1$.

Thus we complete the proof.
\end{proof}

%%%%%%%%%%%%%%%%%%%%%%%%%%%%%%%%%%%%%%%%%%%%%%%%%%%%%%%%%%%%%%%%%%%%%%%%%%%%%%%%%%%%%%%%%%%%%%%%%%%%
%%%%%%%%%%%%%%%%%%%%%%%%%%%%%%%%%%%%%%%%%%%%%%%%%%%%%%%%%%%%%%%%%%%%%%%%%%%%%%%%%%%%%%%%%%%%%%%%%%%%
%%%%%%%%%%%%%%%%%%%%%%%%%%%%%%%%%%%%%%%%%%%%%%%%%%%%%%%%%%%%%%%%%%%%%%%%%%%%%%%%%%%%%%%%%%%%%%%%%%%%
%%%%%%%%%%%%%%%%%%%%%%%%%%%%%%%%%%%%%%%%%%%%%%%%%%%%%%%%%%%%%%%%%%%%%%%%%%%%%%%%%%%%%%%%%%%%%%%%%%%%
%%%%%%%%%%%%%%%%%%%%%%%%%%%%%%%%%%%%%%%%%%%%%%%%%%%%%%%%%%%%%%%%%%%%%%%%%%%%%%%%%%%%%%%%%%%%%%%%%%%%
%%%%%%%%%%%%%%%%%%%%%%%%%%%%%%%%%%%%%%%%%%%%%%%%%%%%%%%%%%%%%%%%%%%%%%%%%%%%%%%%%%%%%%%%%%%%%%%%%%%%
%%%%%%%%%%%%%%%%%%%%%%%%%%%%%%%%%%%%%%%%%%%%%%%%%%%%%%%%%%%%%%%%%%%%%%%%%%%%%%%%%%%%%%%%%%%%%%%%%%%%
%%%%%%%%%%%%%%%%%%%%%%%%%%%%%%%%%%%%%%%%%%%%%%%%%%%%%%%%%%%%%%%%%%%%%%%%%%%%%%%%%%%%%%%%%%%%%%%%%%%%
%%%%%%%%%%%%%%%%%%%%%%%%%%%%%%%%%%%%%%%%%%%%%%%%%%%%%%%%%%%%%%%%%%%%%%%%%%%%%%%%%%%%%%%%%%%%%%%%%%%%
%%%%%%%%%%%%%%%%%%%%%%%%%%%%%%%%%%%%%%%%%%%%%%%%%%%%%%%%%%%%%%%%%%%%%%%%%%%%%%%%%%%%%%%%%%%%%%%%%%%%
\section{Proof of Theorem~\ref{main-3}}\label{thm-3}

Let $F_k$ be the matrix whose $(k,k)$ entry is $-1$, the other diagonal entries are $1$ and the other entries are $0$.
Note that $F_{kl}=F_kF_l$.
$\GH_2(n)$ has a following presentation.

%%%%%%%%%%%%%%%%%%%%%%%%%%%%%%%%%%%%%%%%%%%%%%%%%%%%%%%%%%%%%%%%%%%%%%%%%%%%%%%%%%%%%%%%%%%%%%%%%%%%
\begin{thm}\cite{Fu, Ko1, HK}\label{GH}
For $n\geq2$, $\GH_2(n)$ is generated by $E_{ij}$ and $F_k$ for $1\leq{i,j}\leq{n}$ with $i\neq{}j$ and $1\leq{k}\leq{n}$.
The defining relators are as follows.
\begin{enumerate}
\item[(a)]	$F_i^2$ when $n\geq1$,
\item[(b)]	$(F_iF_j)^2$, $(E_{ij}F_i)^2$, $(E_{ij}F_j)^2$ when $g\geq2$,
\item[(c)]	$[E_{ij},E_{ik}]$, $[E_{ij},E_{kj}]$, $[E_{ij},E_{jk}]E_{ik}^{-2}$, $[E_{ij},F_k]$ when $n\geq3$,
\item[(d)]	$[E_{ij},E_{kl}]$ when $n\geq4$,
\item[(e)]	$(E_{ji}^{-1}E_{ij}E_{kj}^{-1}E_{jk}E_{ik}^{-1}E_{ki})^2$ with $i<j<k$, when $n\geq3$,
\end{enumerate}
where indices $i$, $j$, $k$ and $l$ are distinct.
\end{thm}
%%%%%%%%%%%%%%%%%%%%%%%%%%%%%%%%%%%%%%%%%%%%%%%%%%%%%%%%%%%%%%%%%%%%%%%%%%%%%%%%%%%%%%%%%%%%%%%%%%%%

We now prove Theorem~\ref{main-3}.

\begin{proof}[Proof of Theorem~\ref{main-3}]
In order to prove Theorem~\ref{main-3}, we use the Reidemeister-Schreier method for the finite presentation of $\GH_2(n)$.

For $X\in\GL(n;\Z)$, let $\overline{X}=I$ or $F_1$ if $\det{X}=1$ or $-1$ respectively, and let $[X]=X\overline{X}^{-1}$.
We take $U=\{I,F_1\}$ as a Schreier transversal for $\G_2(n)$ in $\GH_2(n)$.

$\G_2(n)$ is generated by $\left[yx^{\pm1}\right]$ with $yx^{\pm1}\neq\overline{yx^{\pm1}}$ for any generator $x$ of $\GH_2(n)$ in Theorem~\ref{GH} and $y\in{U}$.
We see
\begin{itemize}
\item	$\left[IE_{ij}^{\pm1}\right]=IE_{ij}^{\pm1}I^{-1}=E_{ij}^{\pm1}$,
\item	$\left[IF_k^{\pm1}\right]=IF_kF_1^{-1}=I$ or $F_{1k}$ if $k=1$ or $k\geq2$, respectively,
\item	$\left[F_1E_{ij}^{\pm1}\right]=F_1E_{ij}^{\pm1}F_1^{-1}=E_{ij}^{\mp1}$ or $E_{ij}^{\pm1}$ if either $i$ or $j$ is $1$ or the others, respectively,
\item	$\left[F_1F_k^{\pm1}\right]=F_1F_kI^{-1}=I$ or $F_{1k}$ if $k=1$ or $k\geq2$, respectively.
\end{itemize}
Therefore $\G_2(n)$ is generated by $E_{ij}$ and $F_{1k}$ for $1\leq{i,j}\leq{n}$ with $i\neq{}j$ and $2\leq{k}\leq{n}$.

The defining relators are
$$\left[yr_1\right]\left[\overline{yr_1}r_2\right]\cdots\left[\overline{yr_1r_2\cdots{r_{l-1}}}{r_l}\right],$$
for any relator $r_1r_2\cdots{}r_l$ of $\GH_2(n)$ in Theorem~\ref{GH} and $y\in{U}$.
We see
\begin{enumerate}
\item[(a)]	\begin{itemize}
		\item	$\left[IF_i\right]\left[F_1F_i\right]=I^2$ or $F_{1i}^2$ if $i=1$ or $i\geq2$, respectively,
		\item	$\left[F_1F_i\right]\left[IF_i\right]=I^2$ or $F_{1i}^2$ if $i=1$ or $i\geq2$, respectively,
		\end{itemize}
\item[(b)]	\begin{itemize}
		\item	$\left[IF_i\right]\left[F_1F_j\right]\left[IF_i\right]\left[F_1F_j\right]=(IF_{1j})^2$, $(F_{1i}I)^2$ or $(F_{1i}F_{1j})^2$ if $i=1$, $j=1$ or the others, respectively,
		\item	$\left[F_1F_i\right]\left[IF_j\right]\left[F_1F_i\right]\left[IF_j\right]=(IF_{1j})^2$, $(F_{1i}I)^2$ or $(F_{1i}F_{1j})^2$ if $i=1$, $j=1$ or the others, respectively,
		\item	$\left[IE_{ij}\right]\left[IF_i\right]\left[F_1E_{ij}\right]\left[F_1F_i\right]=E_{1j}IE_{1j}^{-1}I$, $[E_{i1},F_{1i}]$ or $(E_{ij}F_{1i})^2$ if $i=1$, $j=1$ or the others, respectively,
		\item	$\left[F_1E_{ij}\right]\left[F_1F_i\right]\left[IE_{ij}\right]\left[IF_i\right]=E_{1j}^{-1}IE_{1j}I$, $[E_{i1}^{-1},F_{1i}]$ or $(E_{ij}F_{1i})^2$ if $i=1$, $j=1$ or the others, respectively,
		\item	$\left[IE_{ij}\right]\left[IF_j\right]\left[F_1E_{ij}\right]\left[F_1F_j\right]=[E_{1j},F_{1j}]$, $E_{i1}IE_{i1}^{-1}I$ or $(E_{ij}F_{1j})^2$ if $i=1$, $j=1$ or the others respectively,
		\item	$\left[F_1E_{ij}\right]\left[F_1F_j\right]\left[IE_{ij}\right]\left[IF_j\right]=[E_{1j}^{-1},F_{1j}]$, $E_{i1}^{-1}IE_{i1}I$ or $(E_{ij}F_{1j})^2$ if $i=1$, $j=1$ or the others, respectively,
		\end{itemize}
\item[(c)]	\begin{itemize}
		\item	$\left[IE_{ij}\right]\left[IE_{ik}\right]\left[IE_{ij}^{-1}\right]\left[IE_{ik}^{-1}\right]=[E_{ij},E_{ik}]$,
		\item	$\left[F_1E_{ij}\right]\left[F_1E_{ik}\right]\left[F_1E_{ij}^{-1}\right]\left[F_1E_{ik}^{-1}\right]=[E_{1j}^{-1},E_{1k}^{-1}]$, $[E_{i1}^{-1},E_{ik}]$, $[E_{ij},E_{i1}^{-1}]$ or $[E_{ij},E_{ik}]$ if $i=1$, $j=1$, $k=1$ or the others, respectively,
		\item	$\left[IE_{ij}\right]\left[IE_{kj}\right]\left[IE_{ij}^{-1}\right]\left[IE_{kj}^{-1}\right]=[E_{ij},E_{kj}]$,
		\item	$\left[F_1E_{ij}\right]\left[F_1E_{kj}\right]\left[F_1E_{ij}^{-1}\right]\left[F_1E_{kj}^{-1}\right]=[E_{1j}^{-1},E_{kj}]$, $[E_{i1}^{-1},E_{k1}^{-1}]$, $[E_{ij},E_{1j}^{-1}]$ or $[E_{ij},E_{kj}]$ if $i=1$, $j=1$, $k=1$ or the others, respectively,
		\item	$\left[IE_{ij}\right]\left[IE_{jk}\right]\left[IE_{ij}^{-1}\right]\left[IE_{jk}^{-1}\right]\left[IE_{ik}^{-2}\right]=[E_{ij},E_{jk}]E_{ik}^{-2}$,
		\item	$\left[F_1E_{ij}\right]\left[F_1E_{jk}\right]\left[F_1E_{ij}^{-1}\right]\left[F_1E_{jk}^{-1}\right]\left[F_1E_{ik}^{-2}\right]=[E_{1j}^{-1},E_{jk}]E_{1k}^2$, $[E_{i1}^{-1},E_{1k}^{-1}]E_{ik}^{-2}$, $[E_{ij},E_{j1}^{-1}]E_{i1}^2$ or $[E_{ij},E_{jk}]E_{ik}^{-2}$ if $i=1$, $j=1$, $k=1$ or the others, respectively,
		\item	$\left[IE_{ij}\right]\left[IF_k\right]\left[F_1E_{ij}^{-1}\right]\left[F_1F_k^{-1}\right]=(E_{1j}F_{1k})^2$, $(E_{i1}F_{1k})^2$, $E_{ij}IE_{ij}^{-1}I$ or $[E_{ij},F_{1k}]$ if $i=1$, $j=1$, $k=1$ or the others, respectively,
		\item	$\left[F_1E_{ij}\right]\left[F_1F_k\right]\left[IE_{ij}^{-1}\right]\left[IF_k^{-1}\right]=(E_{1j}^{-1}F_{1k})^2$, $(E_{i1}^{-1}F_{1k})^2$, $E_{ij}IE_{ij}^{-1}I$ or $[E_{ij},F_{1k}]$ if $i=1$, $j=1$, $k=1$ or the others, respectively,
		\end{itemize}
\item[(d)]	\begin{itemize}
		\item	$\left[IE_{ij}\right]\left[IE_{kl}\right]\left[IE_{ij}^{-1}\right]\left[IE_{kl}^{-1}\right]=[E_{ij},E_{kl}]$,
		\item	$\left[F_1E_{ij}\right]\left[F_1E_{kl}\right]\left[F_1E_{ij}^{-1}\right]\left[F_1E_{kl}^{-1}\right]=[E_{1j}^{-1},E_{kl}]$, $[E_{i1}^{-1},E_{kl}]$, $[E_{ij},E_{1l}^{-1}]$, $[E_{ij},E_{k1}^{-1}]$ or $[E_{ij},E_{kl}]$ if $i=1$, $j=1$, $k=1$, $l=1$ or the others, respectively,
		\end{itemize}
\item[(e)]	\begin{itemize}
		\item	$\left(\left[IE_{ji}^{-1}\right]\left[IE_{ij}\right]\left[IE_{kj}^{-1}\right]\left[IE_{jk}\right]\left[IE_{ik}^{-1}\right]\left[IE_{ki}\right]\right)^2=(E_{ji}^{-1}E_{ij}E_{kj}^{-1}E_{jk}E_{ik}^{-1}E_{ki})^2$,
		\item	$\left(\left[F_1E_{ji}^{-1}\right]\left[F_1E_{ij}\right]\left[F_1E_{kj}^{-1}\right]\left[F_1E_{jk}\right]\left[F_1E_{ik}^{-1}\right]\left[F_1E_{ki}\right]\right)^2=(E_{j1}E_{1j}^{-1}E_{kj}^{-1}E_{jk}E_{1k}E_{k1}^{-1})^2$ or
		$(E_{ji}^{-1}E_{ij}E_{kj}^{-1}E_{jk}E_{ik}^{-1}E_{ki})^2$ if $i=1$ or the others, respectively.
		On the other hand, the relator $(E_{j1}E_{1j}^{-1}E_{kj}^{-1}E_{jk}E_{1k}E_{k1}^{-1})^2$ is obtained from other relators $(E_{j1}^{-1}E_{1j}E_{kj}^{-1}E_{jk}E_{1k}^{-1}E_{k1})^2$, $[E_{j1},F_{1j}]$, $[E_{1j},F_{1j}]$, $(E_{kj}F_{1j})^2$, $(E_{jk}F_{1j})^2$, $(E_{1k}F_{1j})^2$, $(E_{k1}F_{1j})^2$, $(E_{j1}F_{1k})^2$, $(E_{1j}F_{1k})^2$, $(E_{kj}F_{1k})^2$, $(E_{jk}F_{1k})^2$, $[E_{1k},F_{1k}]$, $[E_{k1},F_{1k}]$, $F_{1j}^2$ and $F_{1k}^2$, as follows.
		$$(E_{j1}E_{1j}^{-1}E_{kj}^{-1}E_{jk}E_{1k}E_{k1}^{-1})^2=F_{1k}F_{1j}(E_{j1}^{-1}E_{1j}E_{kj}^{-1}E_{jk}E_{1k}^{-1}E_{k1})^2F_{1j}F_{1k}.$$
		\end{itemize}
\end{enumerate}
In summary, $\G_2(n)$ has relators
\begin{enumerate}
\item	$F_{1i}^2$, $[E_{1i},F_{1i}]$, $[E_{i1},F_{1i}]$ when $n\geq2$,
\item	$[E_{ij},E_{ik}]$, $[E_{ij},E_{kj}]$, $[E_{ij},E_{jk}]E_{ik}^{-2}$, $(F_{1i}F_{1j})^2$, $(E_{1i}F_{1j})^2$, $(E_{ij}F_{1j})^2$, $(E_{i1}F_{1j})^2$, $(E_{ij}F_{1i})^2$ when $n\geq3$,
\item	$[E_{ij},E_{kl}]$ when $n\geq4$,
\item	$(E_{ji}^{-1}E_{ij}E_{kj}^{-1}E_{jk}E_{ik}^{-1}E_{ki})^2$ with $i<j<k$, when $n\geq3$.
\end{enumerate}

Thus we complete the proof.
\end{proof}

\begin{rem}
Since $\G_d(n)$ is a finite index subgroup of $\SL(n;\Z)$, $\G_d(n)$ can be finitely presented.
In fact, using the Reidemeister-Schreier method, we can obtain a finite presentation for $\G_d(n)$.
In general, for a finitely presented group $G$ and its subgroup $H$, the Reidemeister-Schreier method is useful when $G/H$ is a small group or an abelian group.
By Example~\ref{exam-thm-5}, since $\G_{2^{l-1}}(n)/\G_{2^l}(n)$ is isomorphic to $\left(\Z/2\Z\right)^{n^2-1}$ for $l\ge2$, we will obtain a finite presentation for $\G_{2^l}(n)$ from a presentation for $\G_{2^{l-1}}(n)$, inductively.
\end{rem}

%%%%%%%%%%%%%%%%%%%%%%%%%%%%%%%%%%%%%%%%%%%%%%%%%%%%%%%%%%%%%%%%%%%%%%%%%%%%%%%%%%%%%%%%%%%%%%%%%%%%
%%%%%%%%%%%%%%%%%%%%%%%%%%%%%%%%%%%%%%%%%%%%%%%%%%%%%%%%%%%%%%%%%%%%%%%%%%%%%%%%%%%%%%%%%%%%%%%%%%%%
%%%%%%%%%%%%%%%%%%%%%%%%%%%%%%%%%%%%%%%%%%%%%%%%%%%%%%%%%%%%%%%%%%%%%%%%%%%%%%%%%%%%%%%%%%%%%%%%%%%%
%%%%%%%%%%%%%%%%%%%%%%%%%%%%%%%%%%%%%%%%%%%%%%%%%%%%%%%%%%%%%%%%%%%%%%%%%%%%%%%%%%%%%%%%%%%%%%%%%%%%
%%%%%%%%%%%%%%%%%%%%%%%%%%%%%%%%%%%%%%%%%%%%%%%%%%%%%%%%%%%%%%%%%%%%%%%%%%%%%%%%%%%%%%%%%%%%%%%%%%%%
%%%%%%%%%%%%%%%%%%%%%%%%%%%%%%%%%%%%%%%%%%%%%%%%%%%%%%%%%%%%%%%%%%%%%%%%%%%%%%%%%%%%%%%%%%%%%%%%%%%%
%%%%%%%%%%%%%%%%%%%%%%%%%%%%%%%%%%%%%%%%%%%%%%%%%%%%%%%%%%%%%%%%%%%%%%%%%%%%%%%%%%%%%%%%%%%%%%%%%%%%
%%%%%%%%%%%%%%%%%%%%%%%%%%%%%%%%%%%%%%%%%%%%%%%%%%%%%%%%%%%%%%%%%%%%%%%%%%%%%%%%%%%%%%%%%%%%%%%%%%%%
%%%%%%%%%%%%%%%%%%%%%%%%%%%%%%%%%%%%%%%%%%%%%%%%%%%%%%%%%%%%%%%%%%%%%%%%%%%%%%%%%%%%%%%%%%%%%%%%%%%%
%%%%%%%%%%%%%%%%%%%%%%%%%%%%%%%%%%%%%%%%%%%%%%%%%%%%%%%%%%%%%%%%%%%%%%%%%%%%%%%%%%%%%%%%%%%%%%%%%%%%
\section{Proof of Theorem~\ref{main-4}}\label{thm-4}

We denote the $(i,j)$ entry of a matrix $A$ by $A(i,j)$.
For any $A\in\G_d(2)$, we construct a product $X$ of matrices of Theorem~\ref{main-4} such that $XA(1,1)=1$.
Then, we notice that $e_{21}^{-XA(2,1)}XAe_{12}^{-XA(1,2)}=I$.
We note that
\begin{eqnarray*}
e_{12}^lA&=&\begin{pmatrix}A(1,1)+lA(2,1)&A(1,2)+lA(2,2)\\A(2,1)&A(2,2)\end{pmatrix},\\
e_{21}^lA&=&\begin{pmatrix}A(1,1)&A(1,2)\\A(2,1)+lA(1,1)&A(2,2)+lA(1,2)\end{pmatrix}.
\end{eqnarray*}

We now prove Theorem~\ref{main-4}.

\begin{proof}[Proof of Theorem~\ref{main-4}]
\begin{enumerate}
\item	For any $A\in\G_d(2)$, if $A(1,1)>1$, there are integers $l$ and $0\le{m}\le{d-1}$ such that
	$$-\frac{1}{2}|A(1,1)|\le{A(2,1)+(ld-m)A(1,1)}<\frac{1}{2}|A(1,1)|.$$
	Note that if $d=4$, $\displaystyle-\frac{1}{2}|A(1,1)|\neq{}A(2,1)+(ld-m)A(1,1)$.
	We have
	\begin{eqnarray*}
	(e_{21}^me_{12}^de_{21}^{-m})^\epsilon(e_{21}^d)^lA(1,1)
	&=&
	e_{12}^{\epsilon{}d}e_{21}^{ld-m}A(1,1)\\
	&=&
	A(1,1)+\epsilon{}d(A(2,1)+(ld-m)A(1,1)),
	\end{eqnarray*}
	where $\epsilon$ is $-1$ or $1$ if $A(1,1)$ and $A(2,1)+(ld-m)A(1,1)$ are same sign or different sign respectively.
	Then, we see
	$$|A(1,1)|>|(e_{21}^me_{12}^de_{21}^{-m})^\epsilon(e_{21}^d)^lA(1,1)|.$$
	Similarly, if $(e_{21}^me_{12}^de_{21}^{-m})^\epsilon(e_{21}^d)^lA(1,1)>1$, there are integers $l^\prime$, $0\le{m^\prime}\le{d-1}$ and $\epsilon^\prime=\pm1$ such that
	$$|(e_{21}^me_{12}^de_{21}^{-m})^\epsilon(e_{21}^d)^lA(1,1)|>|(e_{21}^{m^\prime}e_{12}^de_{21}^{-m^\prime})^{\epsilon^\prime}(e_{21}^d)^{l^\prime}(e_{21}^me_{12}^de_{21}^{-m})^\epsilon(e_{21}^d)^lA(1,1)|.$$
	Repeating this operation, we can obtain a product $X$ of matrices of Theorem~\ref{main-4}~(1) such that $XA(1,1)=1$.
	Therefore we obtain the claim.
\item	For any $A\in\G_5(2)$, if $A(1,1)>1$, there are integers $l$ and $0\le{m}\le4$ such that
	$$-\frac{1}{2}|A(1,1)|<{A(2,1)+(5l-m)A(1,1)}<\frac{1}{2}|A(1,1)|.$$
	Note that if we suppose $\displaystyle\frac{1}{2}|A(1,1)|=|A(2,1)+(5l-m)A(1,1)|$, then $A(1,1)$ and $A(2,1)$ are multiples of $\displaystyle\frac{1}{2}|A(1,1)|>1$.
	On the other hand, since $A(1,1)$ and $A(2,1)$ are relatively prime, we have $\displaystyle\frac{1}{2}|A(1,1)|\neq|{A(2,1)+(5l-m)A(1,1)}|$.
	When $\displaystyle|A(2,1)+(5l-m)A(1,1)|<\frac{2}{5}|A(1,1)|$, similar to (1), we see
	$$|A(1,1)|>|(e_{21}^me_{12}^5e_{21}^{-m})^\epsilon(e_{21}^5)^lA(1,1)|$$
	for appropriate $\epsilon=\pm1$.
	When $\displaystyle|A(2,1)+(5l-m)A(1,1)|>\frac{2}{5}|A(1,1)|$, we have
	\begin{eqnarray*}
	(e_{21}^me_{12}^{-2\epsilon}e_{21}^5e_{12}^{2\epsilon}e_{21}^{-m})^\epsilon(e_{21}^5)^lA(1,1)
	&=&
	e_{12}^{-2\epsilon}e_{21}^{5\epsilon}e_{12}^{2\epsilon}e_{21}^{5l-m}A(1,1)\\
	&=&
	A(1,1)+2\epsilon(A(2,1)+(5l-m)A(1,1))\\
	&&
	-2\epsilon((A(2,1)+(5l-m)A(1,1))\\
	&&
	+5\epsilon(A(1,1)+2\epsilon(A(2,1)+(5l-m)A(1,1))))\\
	&=&
	-20\epsilon(A(2,1)+(5l-m)A(1,1))-9A(1,1),
	\end{eqnarray*}
	where $\epsilon$ is $-1$ or $1$ if $A(1,1)$ and $A(2,1)+(5l-m)A(1,1)$ are same sign or different sign respectively, and see
	$$|A(1,1)|>|(e_{21}^me_{12}^{-2\epsilon}e_{21}^5e_{12}^{2\epsilon}e_{21}^{-m})^\epsilon(e_{21}^5)^lA(1,1)|.$$
	Note that $\displaystyle|A(2,1)+(5l-m)A(1,1)|\neq\frac{2}{5}|A(1,1)|$ since $\displaystyle{}A(1,1)$ is not an integer.
	If $(e_{21}^me_{12}^{-2\epsilon}e_{21}^5e_{12}^{2\epsilon}e_{21}^{-m})^\epsilon(e_{21}^5)^lA(1,1)>1$, repeating this operation, we can obtain a product $X$ of matrices of Theorem~\ref{main-4}~(2) such that $XA(1,1)=1$.
	Therefore we obtain the claim.
\item	For any $A\in\G_6(2)$, if $A(1,1)>1$, there are integers $l$ and $0\le{m}\le5$ such that
	$$-\frac{1}{2}|A(1,1)|<{A(2,1)+(6l-m)A(1,1)}<\frac{1}{2}|A(1,1)|.$$
	Note that if we suppose $\displaystyle\frac{1}{2}|A(1,1)|=|A(2,1)+(6l-m)A(1,1)|$, then $A(1,1)$ and $A(2,1)$ are multiples of $\displaystyle\frac{1}{2}|A(1,1)|>1$.
	On the other hand, since $A(1,1)$ and $A(2,1)$ are relatively prime, we have $\displaystyle\frac{1}{2}|A(1,1)|\neq|{A(2,1)+(6l-m)A(1,1)}|$.
	When $\displaystyle|A(2,1)+(6l-m)A(1,1)|<\frac{1}{3}|A(1,1)|$, similar to (1) and (2), we see
	$$|A(1,1)|>|(e_{21}^me_{12}^6e_{21}^{-m})^\epsilon(e_{21}^6)^lA(1,1)|$$
	for appropriate $\epsilon=\pm1$.
	When $\displaystyle|A(2,1)+(6l-m)A(1,1)|>\frac{1}{3}|A(1,1)|$, we have
	\begin{eqnarray*}
	e_{21}^m[e_{21}^{3\epsilon},e_{12}^{2\epsilon}]e_{21}^{-m}(e_{21}^6)^lA(1,1)
	&=&
	e_{12}^{2\epsilon}e_{21}^{-3\epsilon}e_{12}^{-2\epsilon}e_{21}^{6l-m}A(1,1)\\
	&=&
	A(1,1)-2\epsilon(A(2,1)+(6l-m)A(1,1))\\
	&&
	+2\epsilon((A(2,1)+(6l-m)A(1,1))\\
	&&
	-3\epsilon(A(1,1)-2\epsilon(A(2,1)+(6l-m)A(1,1))))\\
	&=&
	12\epsilon(A(2,1)+(6l-m)A(1,1))-5A(1,1),
	\end{eqnarray*}
	where $\epsilon$ is $1$ or $-1$ if $A(1,1)$ and $A(2,1)+(6l-m)A(1,1)$ are same sign or different sign respectively, and see
	$$|A(1,1)|>|e_{21}^m[e_{21}^{3\epsilon},e_{12}^{2\epsilon}]e_{21}^{-m}(e_{21}^6)^lA(1,1)|.$$
	Note that $\displaystyle|A(2,1)+(6l-m)A(1,1)|\neq\frac{1}{3}|A(1,1)|$ since $\displaystyle\frac{1}{3}A(1,1)$ is not an integer.
	If $e_{21}^m[e_{21}^{3\epsilon},e_{12}^{2\epsilon}]e_{21}^{-m}(e_{21}^6)^lA(1,1)>1$, repeating this operation, we can obtain a product $X$ of matrices of Theorem~\ref{main-4}~(3) such that $XA(1,1)=1$.
	Therefore we obtain the claim.
\end{enumerate}
Thus we complete the proof.
\end{proof}

\begin{rem}
From the proof of Proposition~\ref{[G,G]}, we have  $[\GH_{d_1}(2),\GH_{d_2}(2)]\subset\G_{d_1d_2}(2)$ for $d_1\ge1$ and $d_2\ge1$.
In addition, since $e_{ij}^{2l}=[e_{ij}^l,F_i]\in[\G_l(n),\GH_2(n)]$, where $F_i$ is defined at the beginning of Section~\ref{thm-2}, by Theorem~\ref{main-4}, we have the followings.
\begin{itemize}
\item	$[\GH_2(2),\G_2(2)]=[\GH_2(2),\GH_2(2)]=\G_4(2)$.
\item	$[\GH_2(2),\G_3(2)]=\G_6(2)$.
\end{itemize}
\end{rem}

%%%%%%%%%%%%%%%%%%%%%%%%%%%%%%%%%%%%%%%%%%%%%%%%%%%%%%%%%%%%%%%%%%%%%%%%%%%%%%%%%%%%%%%%%%%%%%%%%%%%
%%%%%%%%%%%%%%%%%%%%%%%%%%%%%%%%%%%%%%%%%%%%%%%%%%%%%%%%%%%%%%%%%%%%%%%%%%%%%%%%%%%%%%%%%%%%%%%%%%%%
%%%%%%%%%%%%%%%%%%%%%%%%%%%%%%%%%%%%%%%%%%%%%%%%%%%%%%%%%%%%%%%%%%%%%%%%%%%%%%%%%%%%%%%%%%%%%%%%%%%%
%%%%%%%%%%%%%%%%%%%%%%%%%%%%%%%%%%%%%%%%%%%%%%%%%%%%%%%%%%%%%%%%%%%%%%%%%%%%%%%%%%%%%%%%%%%%%%%%%%%%
%%%%%%%%%%%%%%%%%%%%%%%%%%%%%%%%%%%%%%%%%%%%%%%%%%%%%%%%%%%%%%%%%%%%%%%%%%%%%%%%%%%%%%%%%%%%%%%%%%%%
%%%%%%%%%%%%%%%%%%%%%%%%%%%%%%%%%%%%%%%%%%%%%%%%%%%%%%%%%%%%%%%%%%%%%%%%%%%%%%%%%%%%%%%%%%%%%%%%%%%%
%%%%%%%%%%%%%%%%%%%%%%%%%%%%%%%%%%%%%%%%%%%%%%%%%%%%%%%%%%%%%%%%%%%%%%%%%%%%%%%%%%%%%%%%%%%%%%%%%%%%
%%%%%%%%%%%%%%%%%%%%%%%%%%%%%%%%%%%%%%%%%%%%%%%%%%%%%%%%%%%%%%%%%%%%%%%%%%%%%%%%%%%%%%%%%%%%%%%%%%%%
%%%%%%%%%%%%%%%%%%%%%%%%%%%%%%%%%%%%%%%%%%%%%%%%%%%%%%%%%%%%%%%%%%%%%%%%%%%%%%%%%%%%%%%%%%%%%%%%%%%%
%%%%%%%%%%%%%%%%%%%%%%%%%%%%%%%%%%%%%%%%%%%%%%%%%%%%%%%%%%%%%%%%%%%%%%%%%%%%%%%%%%%%%%%%%%%%%%%%%%%%
\section{Proof of Theorem~\ref{main-5}}\label{thm-5}

First, we show the following lemma.

\begin{lem}\label{lm}
For $l\ge1$, $m\ge1$ and $n\ge3$, we have
\begin{enumerate}
\item	$\G_l(n)\cap\G_m(n)=\G_{\mathrm{lcm}(l,m)}(n)$,
\item	$\G_l(n)\G_m(n)=\G_{\gcd(l,m)}(n)$.
\end{enumerate}
\end{lem}

\begin{proof}
\begin{enumerate}
\item	Since $\G_{\mathrm{lcm}(l,m)}(n)$ is a subgroup of $\G_l(n)$ and $\G_m(n)$, it is clear that $\G_l(n)\cap\G_m(n)\supset\G_{\mathrm{lcm}(l,m)}(n)$.
	For any $X\in\G_l(n)\cap\G_m(n)$, since $X\equiv{I}\pmod{l}$ and $X\equiv{I}\pmod{m}$, we have $X\equiv{I}\pmod{\mathrm{lcm}(l,m)}$, and so  $\G_l(n)\cap\G_m(n)\subset\G_{\mathrm{lcm}(l,m)}(n)$.
\item	Since $\G_l(n)$ and $\G_m(n)$ are subgroups of $\G_{\gcd(l,m)}(n)$, it is clear that  $\G_l(n)\G_m(n)\subset\G_{\gcd(l,m)}(n)$.
	For any normal generator $e_{ij}^{\gcd(l,m)}\in\G_{\gcd(l,m)}(n)$, there are integers $a$ and $b$ such that $al+bm=\gcd(l,m)$.
	Hence we have $e_{ij}^{\gcd(l,m)}=e_{ij}^{al+bm}=(e_{ij}^l)^a(e_{ij}^m)^b\in\G_l(n)\G_m(n)$, and so $\G_l(n)\G_m(n)\supset\G_{\gcd(l,m)}(n)$.
\end{enumerate}
Thus we get the claim.
\end{proof}

We now prove Theorem~\ref{main-5}.

\begin{proof}[Proof of Theorem~\ref{main-5}]
\begin{enumerate}
\item	By Lemma~\ref{lm}, we calculate
	\begin{eqnarray*}
	\G_{\gcd(l,m)}(n)/\G_m(n)
	&=&\G_l(n)\G_m(n)/\G_m(n)\\
	&\cong&\G_l(n)/(\G_l(n)\cap\G_m(n))\\
	&=&\G_l(n)/\G_{\mathrm{lcm}(l,m)}(n).
	\end{eqnarray*}
	Therefore we obtain the claim.
\item	By Lemma~\ref{lm} and (1), we calculate
	\begin{eqnarray*}
	\G_{\gcd(l,m)}(n)/\G_{\mathrm{lcm}(l,m)}(n)
	&=&\G_l(n)\G_m(n)/(\G_l(n)\cap\G_m(n))\\
	&\cong&\G_l(n)\G_m(n)/\G_l(n)\times\G_l(n)\G_m(n)/\G_m(n)\\
	&\cong&\G_{\gcd(l,m)}(n)/\G_l(n)\times\G_{\gcd(l,m)}(n)/\G_m(n)\\
	&\cong&\G_l(n)/\G_{\mathrm{lcm}(l,m)}(n)\times\G_m(n)/\G_{\mathrm{lcm}(l,m)}(n).
	\end{eqnarray*}
	Therefore we obtain the claim.
\item	By the condition, we have a short exact sequence
	$$1\to\G_m(n)/\G_{l^2}(n)\to\G_l(n)/\G_{l^2}(n)\to\G_l(n)/\G_m(n)\to1.$$
	By Proposition~\ref{[G,G]} and Theorem~\ref{main-1}, $\G_l(n)/\G_{l^2}(n)$ is isomorphic to $\left(\Z/l\Z\right)^{n^2-1}$.
	Let
	\begin{eqnarray*}
	X&=&\{e_{ij}^l,e_{k1}e_{1k}^le_{k1}^{-1}\mid1\le{i,j\le{n}~\textrm{with}~i\neq{}j,~2\le{k}\le{n}}\},\\
	Y&=&\{e_{ij}^m,e_{k1}e_{1k}^me_{k1}^{-1}\mid1\le{i,j\le{n}~\textrm{with}~i\neq{}j,~2\le{k}\le{n}}\}.
	\end{eqnarray*}
	Since $\G_l(n)/\G_m(n)$ and $\G_m(n)/\G_{l^2}(n)$ are $\Z/\frac{m}{l}\Z$ module and $\Z/\frac{l^2}{m}\Z$ module respectively, there are natural surjections
	\begin{eqnarray*}
	&&\left(\Z/\frac{m}{l}\Z\right)^{n^2-1}\cong\bigoplus_{x\in{X}}\Z/\frac{m}{l}\Z[x]\to\G_l(n)/\G_m(n),\\
	&&\left(\Z/\frac{l^2}{m}\Z\right)^{n^2-1}\cong\bigoplus_{y\in{Y}}\Z/\frac{l^2}{m}\Z[y]\to\G_m(n)/\G_{l^2}(n).
	\end{eqnarray*}
	Hence we have
	$$\left|\left(\Z/\frac{m}{l}\Z\right)^{n^2-1}\right|\ge\left|\G_l(n)/\G_m(n)\right|=\left|\frac{\G_l(n)/\G_{l^2}(n)}{\G_m(n)/\G_{l^2}(n)}\right|\ge\frac{\left|\left(\Z/l\Z\right)^{n^2-1}\right|}{\left|\left(\Z/\frac{l^2}{m}\Z\right)^{n^2-1}\right|}=\left|\left(\Z/\frac{m}{l}\Z\right)^{n^2-1}\right|,$$
	and so $\G_l(n)/\G_m(n)$ is isomorphic to $\left(\Z/\frac{m}{l}\Z\right)^{n^2-1}$.
	Therefore we obtain the claim.
\end{enumerate}
Thus we complete the proof.
\end{proof}

Finally, here are interesting examples of Theorem~\ref{main-5}.

\begin{exam}\label{exam-thm-5}
\begin{enumerate}
\item	When $l$ and $m$ are relatively prime, we have
	\begin{eqnarray*}
	\G_l(n)/\G_{lm}(n)
	&=&\G_l(n)/\G_{\mathrm{lcm}(l,m)}(n)\\
	&\cong&\G_{\gcd(l,m)}(n)/\G_m(n)\\
	&=&\G_1(n)/\G_m(n)\\
	&\cong&\SL(n;\Z/m\Z).
	\end{eqnarray*}
\item	When $l$ and $m$ are relatively prime, we have
	\begin{eqnarray*}
	\SL(n;\Z/lm\Z)
	&\cong&\G_1(n)/\G_{lm}(n)\\
	&=&\G_{\gcd(l,m)}(n)/\G_{\mathrm{lcm}(l,m)}(n)\\
	&\cong&\G_{\gcd(l,m)}(n)/\G_l(n)\times\G_{\gcd(l,m)}(n)/\G_m(n)\\
	&=&\G_1(n)/\G_l(n)\times\G_1(n)/\G_m(n)\\
	&\cong&\SL(n;\Z/l\Z)\times\SL(n;\Z/m\Z).
	\end{eqnarray*}
\item	When $l\equiv0\pmod{m}$, since $l\mid{lm}$ and $lm\mid{l^2}$, we have
	$$\G_l(n)/\G_{lm}(n)\cong\left(\Z/m\Z\right)^{n^2-1}.$$
	For example,
	\begin{eqnarray*}
	\G_{2l}(n)/\G_{4l}(n)&\cong&\left(\Z/2\Z\right)^{n^2-1},\\
	\G_{3l}(n)/\G_{9l}(n)&\cong&\left(\Z/3\Z\right)^{n^2-1},\\
	\G_{4l}(n)/\G_{16l}(n)&\cong&\left(\Z/4\Z\right)^{n^2-1},\\
	&\vdots&
	\end{eqnarray*}
\end{enumerate}
\end{exam}

%%%%%%%%%%%%%%%%%%%%%%%%%%%%%%%%%%%%%%%%%%%%%%%%%%%%%%%%%%%%%%%%%%%%%%%%%%%%%%%%%%%%%%%%%%%%%%%%%%%%
%%%%%%%%%%%%%%%%%%%%%%%%%%%%%%%%%%%%%%%%%%%%%%%%%%%%%%%%%%%%%%%%%%%%%%%%%%%%%%%%%%%%%%%%%%%%%%%%%%%%
%%%%%%%%%%%%%%%%%%%%%%%%%%%%%%%%%%%%%%%%%%%%%%%%%%%%%%%%%%%%%%%%%%%%%%%%%%%%%%%%%%%%%%%%%%%%%%%%%%%%
%%%%%%%%%%%%%%%%%%%%%%%%%%%%%%%%%%%%%%%%%%%%%%%%%%%%%%%%%%%%%%%%%%%%%%%%%%%%%%%%%%%%%%%%%%%%%%%%%%%%
%%%%%%%%%%%%%%%%%%%%%%%%%%%%%%%%%%%%%%%%%%%%%%%%%%%%%%%%%%%%%%%%%%%%%%%%%%%%%%%%%%%%%%%%%%%%%%%%%%%%
%%%%%%%%%%%%%%%%%%%%%%%%%%%%%%%%%%%%%%%%%%%%%%%%%%%%%%%%%%%%%%%%%%%%%%%%%%%%%%%%%%%%%%%%%%%%%%%%%%%%
%%%%%%%%%%%%%%%%%%%%%%%%%%%%%%%%%%%%%%%%%%%%%%%%%%%%%%%%%%%%%%%%%%%%%%%%%%%%%%%%%%%%%%%%%%%%%%%%%%%%
%%%%%%%%%%%%%%%%%%%%%%%%%%%%%%%%%%%%%%%%%%%%%%%%%%%%%%%%%%%%%%%%%%%%%%%%%%%%%%%%%%%%%%%%%%%%%%%%%%%%
%%%%%%%%%%%%%%%%%%%%%%%%%%%%%%%%%%%%%%%%%%%%%%%%%%%%%%%%%%%%%%%%%%%%%%%%%%%%%%%%%%%%%%%%%%%%%%%%%%%%
%%%%%%%%%%%%%%%%%%%%%%%%%%%%%%%%%%%%%%%%%%%%%%%%%%%%%%%%%%%%%%%%%%%%%%%%%%%%%%%%%%%%%%%%%%%%%%%%%%%%
\section*{Acknowledgements}

The authors would like to express their gratitude to Susumu Hirose and Genki Omori for their valuable comments and meaningful discussions.
The authors also would like to express their gratitude to Andrew Putman for providing useful information on our research.

%%%%%%%%%%%%%%%%%%%%%%%%%%%%%%%%%%%%%%%%%%%%%%%%%%%%%%%%%%%%%%%%%%%%%%%%%%%%%%%%%%%%%%%%%%%%%%%%%%%%
%%%%%%%%%%%%%%%%%%%%%%%%%%%%%%%%%%%%%%%%%%%%%%%%%%%%%%%%%%%%%%%%%%%%%%%%%%%%%%%%%%%%%%%%%%%%%%%%%%%%
%%%%%%%%%%%%%%%%%%%%%%%%%%%%%%%%%%%%%%%%%%%%%%%%%%%%%%%%%%%%%%%%%%%%%%%%%%%%%%%%%%%%%%%%%%%%%%%%%%%%
%%%%%%%%%%%%%%%%%%%%%%%%%%%%%%%%%%%%%%%%%%%%%%%%%%%%%%%%%%%%%%%%%%%%%%%%%%%%%%%%%%%%%%%%%%%%%%%%%%%%
%%%%%%%%%%%%%%%%%%%%%%%%%%%%%%%%%%%%%%%%%%%%%%%%%%%%%%%%%%%%%%%%%%%%%%%%%%%%%%%%%%%%%%%%%%%%%%%%%%%%

\end{document}